\documentclass[11pt]{article}

\usepackage{fullpage}
\usepackage{amsmath}
\usepackage{amssymb}
\usepackage{amsthm}
\usepackage{graphicx}
\usepackage{color}
\usepackage{mathdots}
\usepackage{mathtools}

\DeclareMathOperator*{\argmax}{argmax}

\DeclareMathOperator*{\rank}{rank}

\newcommand{\reals}{\mathbb{R}}
\newcommand{\Sn}{{\mathcal S}(n)}
\newcommand{\Ln}{{\mathcal L}_n}

\newtheorem{theorem}{Theorem}
\newtheorem{lemma}[theorem]{Lemma}

\newtheorem{proposition}[theorem]{Proposition}

\newtheorem{definition}[theorem]{Definition}
\newtheorem{example}[theorem]{Example}

\begin{document}

\title{Iterated Linear Optimization}
\author{Pedro Felzenszwalb \and Caroline Klivans \and Alice Paul}

\maketitle

\begin{abstract}
We introduce a fixed point iteration process built on optimization of
a linear function over a compact domain.  We prove the process always
converges to a fixed point and explore the set of fixed points in
various convex sets.  In particular, we consider elliptopes and derive
an algebraic characterization of their fixed points.  We show that the
attractive fixed points of an elliptope are exactly its vertices.
Finally, we discuss how fixed point iteration can be used for rounding
the solution of a semidefinite programming relaxation.
\end{abstract}



\section{Introduction}
We introduce a fixed point iteration process built on optimization of
a linear function over a compact domain.  Given a domain $\Delta
\subset \mathbb{R}^n$, the process generates a sequence $\{x_0, x_1,
x_2, \ldots\}$, $x_i \in \Delta$, where $x_{i+1}$ maximizes a linear
function defined by $x_i$.  The iteration is guaranteed to converge
for all compact domains, as shown in Theorem~\ref{thm:converges}.  We
focus on convex sets, both polyhedral and smooth.  The fixed points of
linear optimization reflect interesting geometric properties of the
underlying convex set.  Fixed point iteration is a common methodology
in numerical analysis and optimization (see, e.g., \cite{fp,
  fpinverse}).  Moreover, fixed point iteration defines a discrete
dynamical system (\cite{Galor,Holmgren}).  There are several notions
of stability for such systems, and we consider the attractive and
repulsive fixed points of various sets.

We focus in particular on the set of fixed points of linear
optimization in elliptopes.  Elliptopes are a family of convex bodies
that arise naturally in semidefinite programming (SDP) relaxations of
combinatorial optimization problems (see, e.g.,
\cite{Goemans,Laurent,Whatis,Alizadeh}).  A key step in using a convex
relaxation for combinatorial optimization involves rounding.  A
solution found in the relaxed convex body must be \emph{rounded} to a
discrete solution that satisfies the initial combinatorial problem.
For the elliptope, combinatorial solutions correspond precisely to its
vertices.  All points in the elliptope correspond to a positive
semidefinite matrix of a certain form.  In Theorem~\ref{thm:mfp}, we
use this perspective to prove an algebraic characterization of the
fixed points of iterated linear optimization in the elliptope.
Furthermore, we show that the vertices of the elliptope are exactly
the attractive fixed points of our iteration process,
Theorem~\ref{thm:attractive}.  Each step of fixed point iteration
solves a relaxation to the closest vertex problem.  By iterating the
process we obtain a deterministic method for rounding the solution of
an SDP relaxation.

The problem of rounding the solution of an SDP has fundamental
applications in combinatorial optimization
(\cite{Goemans,Frieze,Raghavendra,Barak}).  In a companion paper
(\cite{clustering}) we apply the fixed point iteration process to
clustering.  The approach is based on the classical SDP relaxation for
$k$-way max-cut defined in \cite{Frieze}, combined with iterated
linear optimization for rounding.

\section{Iterated Optimization}

Let $\Delta \subset \reals^n$ be a compact convex set containing the
origin.

Let $T$ be the map defined by linear optimization over $\Delta$,
$$\textrm{for } x \in \reals^n, \, \, \, \,  T(x) = \argmax_{y \in \Delta} \, \, x \cdot y.$$

We consider the process of fixed point iteration with $T$.  That is,
we are interested in sequences $(x_0,x_1,\ldots)$ such that
$$x_{i+1} = T(x_i).$$

Note that the $\argmax$ in the definition of $T$ may not be unique.
In this case $T(x)$ is set valued.  When we write $x_{i+1} =
T(x_i)$ we allow $x_{i+1}$ to be \emph{any} element of $T(x_i)$.  

The fixed point iteration process can be seen as an iterative method
to maximize $f(x) = \frac{1}{2} ||x||^2$ over a convex domain.  To see
this let $g(x) = f(x_i) + \nabla f(x_i) \cdot (x-x_i)$.  The function
$g(x)$ lower-bounds $f(x)$ and the two functions coincide at $x_i$.
Since $\nabla f(x_i) = x_i$ we see that $x_{i+1}$ maximizes $g(x)$.
Therefore $f(x_{i+1}) \ge g(x_{i+1}) \ge g(x_i) = f(x_i)$.  The use of
a linear approximation in each step often leads to an optimization
problem that can be solved efficiently using interior point methods
and related techniques. Moreover, while $f(x)$ may have many global
maxima, the fixed point iteration process can be used to find a
maximum that is ``near'' an initial point (see Section
~\ref{sec:elliptopeITER}).

The interpretation of fixed point iteration with $T$ as a method to
maximize $f(x) = \frac{1}{2} ||x||^2$ is related to the Frank-Wolfe
method \cite{Frank}, although the Frank-Wolfe algorithm is normally
used to \emph{minimize} a convex function over a convex domain.

\subsection{Convergence}

We first prove that iteration with $T$ converges to a set of fixed
points.  While there are many general results about the convergence of
fixed point iteration (\cite{fp}), these results do not apply to our
setting because $T$ is neither contractive nor continuous.

\begin{theorem}
  Let $\{x_i\}$ be a sequence generated by iteration with $T$.  Then
  $\{x_i\}$ has at least one limit point.  If the sequence has more
  than one limit point the set of limit points is connected.
  Moreover, every limit point is a fixed point of $T$.
\label{thm:converges} 
\end{theorem}
\begin{proof}
Let $\{x_i\}$ be a sequence where $x_0 \in \Delta$ is an arbitrary
starting point and $\forall i $ $ x_{i+1} = T(x_i)$.  Again, note that
the map $T$ may be set valued, and we allow for any choice of $x_{i+1}
\in T(x_i)$ at each stage of the iteration.

By definition of $T$ we have
\begin{equation}
  x_{i+1}\cdot x_i \ge x_i \cdot x_i.
  \label{eqn:step1}
\end{equation}

Using $||x_{i+1}-x_i||^2  \ge 0$ and (\ref{eqn:step1}) gives
\begin{equation}
  x_{i+1} \cdot x_{i+1} \ge x_{i+1} \cdot x_i.
  \label{eqn:step2}
\end{equation}

Let
$$a_i = x_i \cdot x_i = ||x_i||^2.$$ Together (\ref{eqn:step1}) and
(\ref{eqn:step2}) imply that $a_{i} \le a_{i+1}$.  Since the sequence
$\{a_i\}$ is non-decreasing and bounded it converges.  Let $a =
\lim_{i \rightarrow \infty} a_i$.

Since the sequence $\{x_i\}$ is bounded there is a subsequence of
$\{x_i\}$ that converges.  Therefore the sequence $\{x_i\}$ has at
least one limit point.

Using (\ref{eqn:step1}) we see that
\begin{align*}
  ||x_{i+1}-x_i||^2 & = x_{i+1} \cdot x_{i+1} + x_i \cdot x_i - 2 x_{i+1} \cdot x_i, \\
  & \le x_{i+1} \cdot x_{i+1} - x_i \cdot x_i, \\
  & = a_{i+1} - a_i.
\end{align*}
Therefore $\lim_{i \rightarrow \infty} ||x_{i+1}-x_{i}||=0$.  This
implies the sequence $\{x_i\}$ has a single limit point or the set of
limit points is connected (see, e.g., \cite{Asic}).

Now let $x^*$ be a limit point of $\{x_i\}$.  We claim any such $x^*$
is a fixed point.

Suppose $x^*$ is not a fixed point.  Then there must exist $y \in
\Delta$ such that $y \cdot x^* > x^* \cdot x^*$.  Since there is a
subsequence of $\{x_i\}$ that converges to $x^*$, there is an element
$x_i$ that is sufficiently close to $x^*$ such that $y \cdot x_i > x^*
\cdot x^*$.  Since $x_{i+1} = T(x_i)$, $$x_{i+1} \cdot x_i \ge y \cdot
x_i$$ and
$$x_{i+1} \cdot x_{i+1} \ge x_{i+1} \cdot x_{i} \ge y \cdot x_i > x^* \cdot x^* = a \ge x_{i+1} \cdot x_{i+1},$$ which is a contradiction.
\end{proof}

Note that if the set of fixed points in $\Delta$ is finite then any
sequence $\{x_i\}$ generated by $T$ converges to a single fixed point
$x^*$.  This follows from the fact that the set of limit points is
connected.

In this paper we focus on convex spaces $\Delta$.  Note, however, that
the above proof only uses compactness and not convexity.

\section{Fixed Points}

A \emph{fixed point} of $\Delta$ is a point $x \in \Delta$ such that
$x \in T(x)$.  Geometrically, fixed points can be described in terms
of normal cones.

For a point $x \in \Delta$, the \emph{normal cone of $\Delta$ at $x$}
is the set
$$N(\Delta,x) = \{ y \in \reals^n \,|\, y \cdot x \ge y \cdot z \;\; \forall z \in \Delta \}.$$

Note that $x \in T(x)$, i.e. $x$ is a fixed point,  exactly when $x \in N(\Delta,x)$.

The fixed points are distinguished boundary points of $\Delta$
that can be of significant interest.  

\begin{example}[Elliptope]
Figure~\ref{fig:elliptope} illustrates the fixed points of $T$ in the
elliptope ${\cal L}_3$, a convex shape that arises in the SDP
relaxation of max-cut and various other combinatorial optimization
problems.  In this $3$-dimensional example the fixed points of $T$
include both the vertices of the convex shape and several other
distinguished points.  We will analyze the fixed points of the
elliptope in arbitrary dimensions in Section~\ref{sec:elliptope}.
\end{example}

\begin{figure}
  \centering
  \includegraphics[width=3in]{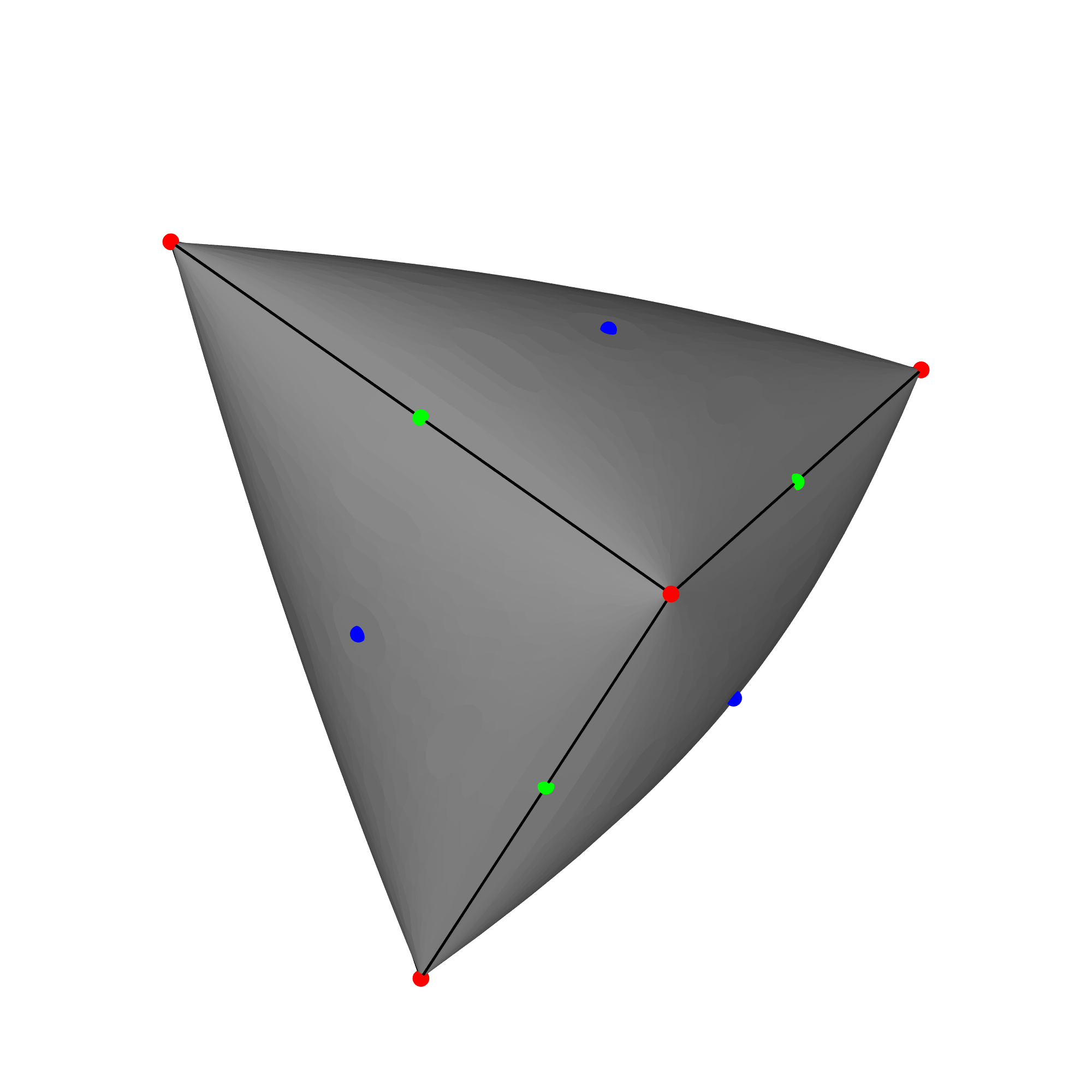}
  \caption{The elliptope ${\cal L}_3$.  The highlighted points are the
    fixed points of $T$.  Points in the elliptope correspond to
    certain positive semidefinite matrices (see
    Section~\ref{sec:elliptope}).  The red fixed points are
    irreducible matrices with rank 1, the blue fixed points are
    irreducible matrices with rank 2 and the green fixed points are
    reducible matrices with rank 2.}
  \label{fig:elliptope}
\end{figure}

Let $n(x)$ be the normal direction at a smooth boundary point $x \in
\Delta$.  Then $x \in N(\Delta,x)$ when $x = \lambda n(x)$.  In
particular, $x$ is a fixed point exactly when the line through $x$
with direction $n(x)$ includes the origin.

\begin{example}[Off-centered disk]
\label{ex:disk}
Figure~\ref{fig:circle} shows an example where $\Delta$ is an
off-center disk.  The disk contains the origin but not in its center.
In this case there are two fixed points $A$ and $R$, where the line
defined by the center of the disk and the origin crosses the boundary.
Figure~\ref{fig:circle}(a) illustrates the computation of $T(x)$ as
the boundary point in a tangent line perpendicular to $x$.
Figure~\ref{fig:circle}(b) shows the result of fixed point iteration
$x_{i+1} = T(x_i)$ starting at $x_0$.  Note how $x_0$ is near one
fixed point ($R$) but the iteration converges to the other fixed point
($A$).  In this case $A$ is an attractive fixed point, while $R$ is
repelling (see Section~\ref{sec:types}).  If we start the iteration
anywhere except at $R$ the process converges to $A$.  Note that if the
origin is at the center of the disk then all boundary points are fixed
points (neither attractive nor repelling).
\end{example}

\begin{example}[Ellipse]
\label{ex:ellipse}
Figure~\ref{fig:ellipse} shows an example where $\Delta$ is an ellipse
centered at the origin.  In this case there are two attractive ($A_1,
A_2$) and two repelling ($R_1, R_2$) fixed points.  Each attractive
fixed point is in the major axis of the ellipse.
The repelling fixed points are in the minor axis.
\end{example}

\begin{figure}
  \centerline{\begin{tabular}{cc}
  \includegraphics[height=3in]{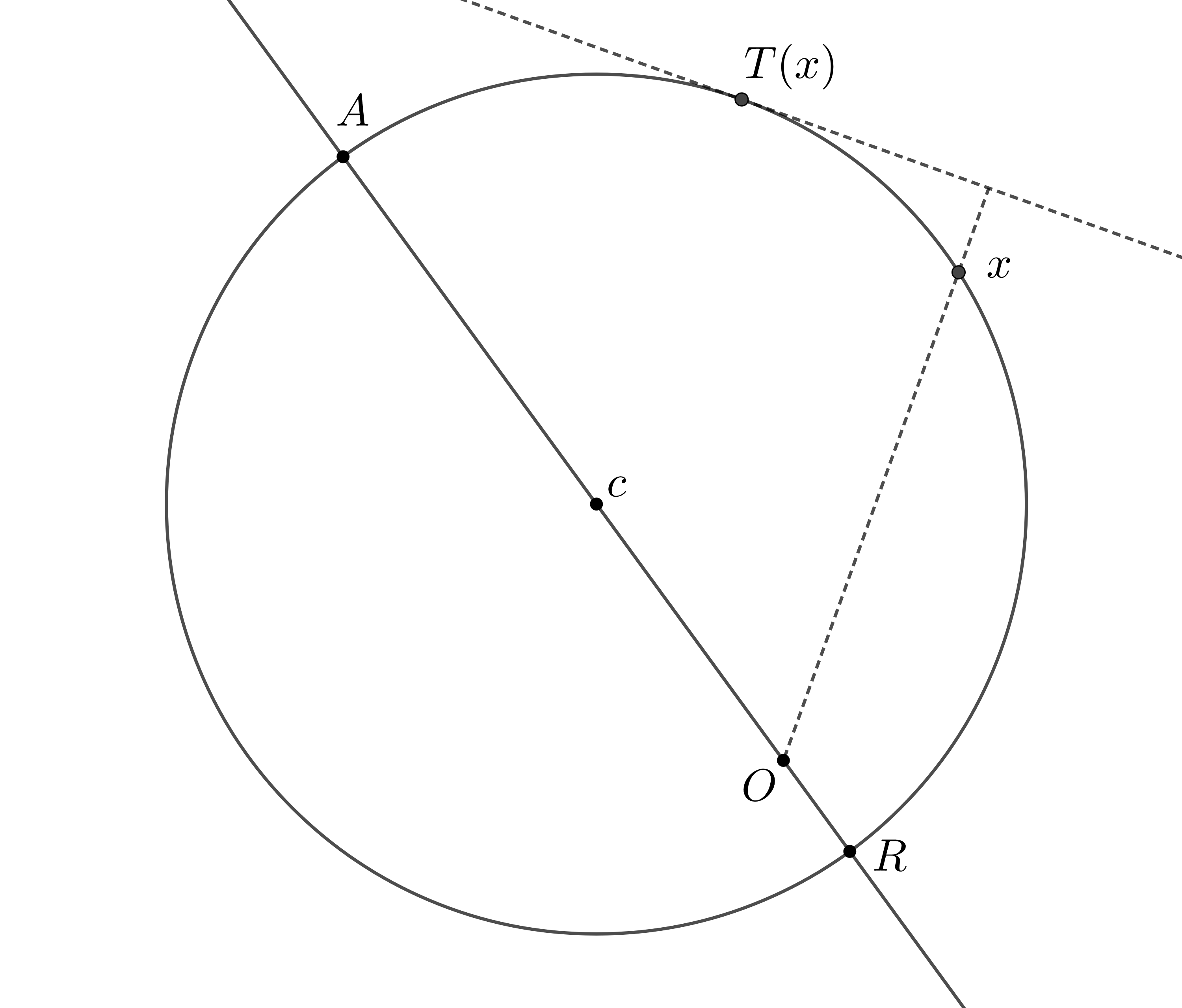} & \includegraphics[height=3in]{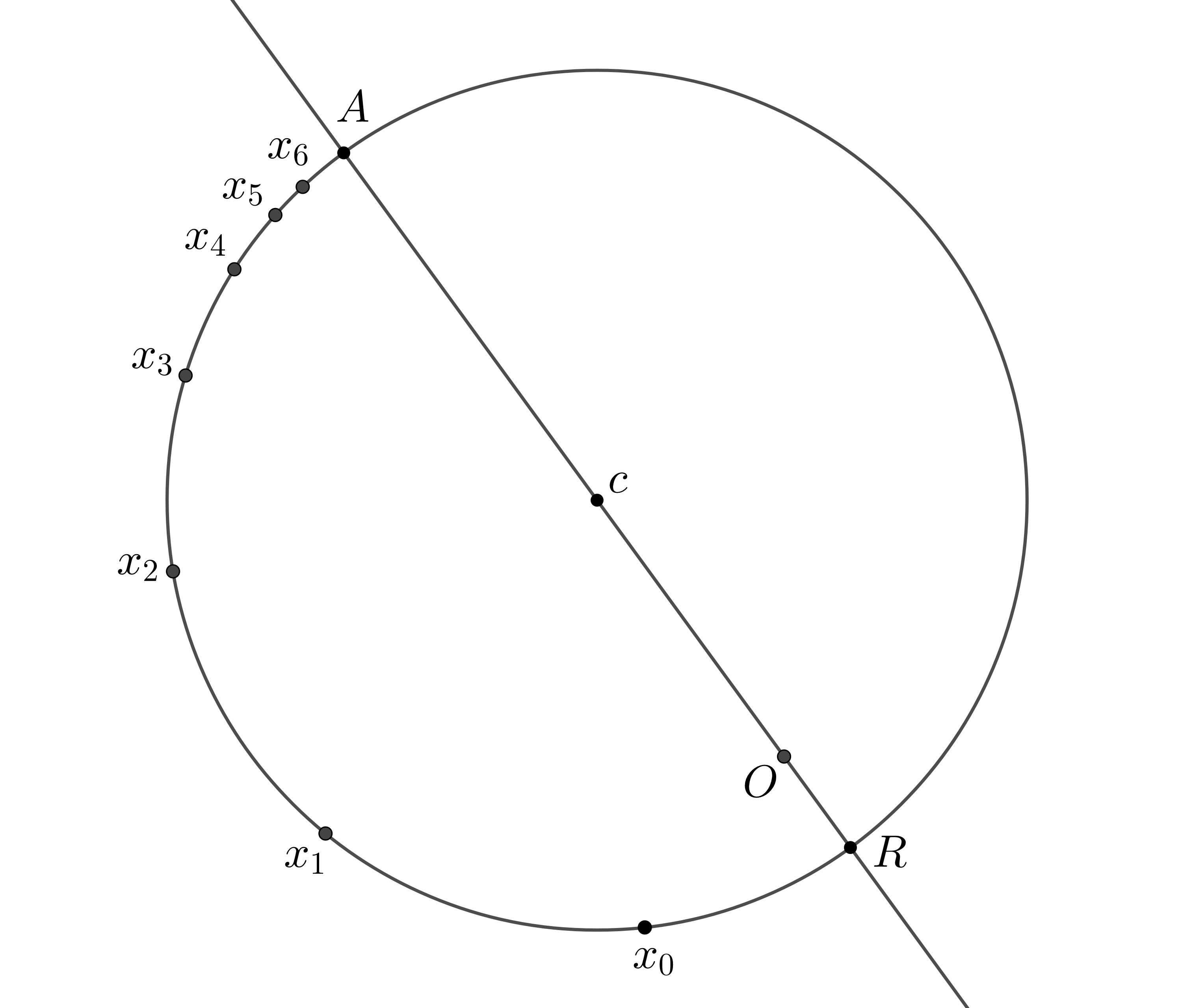} \\
  (a) & (b)
  \end{tabular}}
  \caption{(a) Illustration of the map $T$ and (b) fixed point
    iteration when $\Delta$ is an off-center disk.  Here $c$ is the
    center of the disk while $O$ is the origin.  The points $A$ and
    $R$ are the fixed points of $T$.  Fixed point iteration converges
    to $A$ even when we start from a point closer to $R$.  In this
    case $A$ is an attractive fixed point, while $R$ is a repelling
    fixed point.}
    \label{fig:circle}
\end{figure}

\begin{figure}
      \centerline{\includegraphics[height=3in]{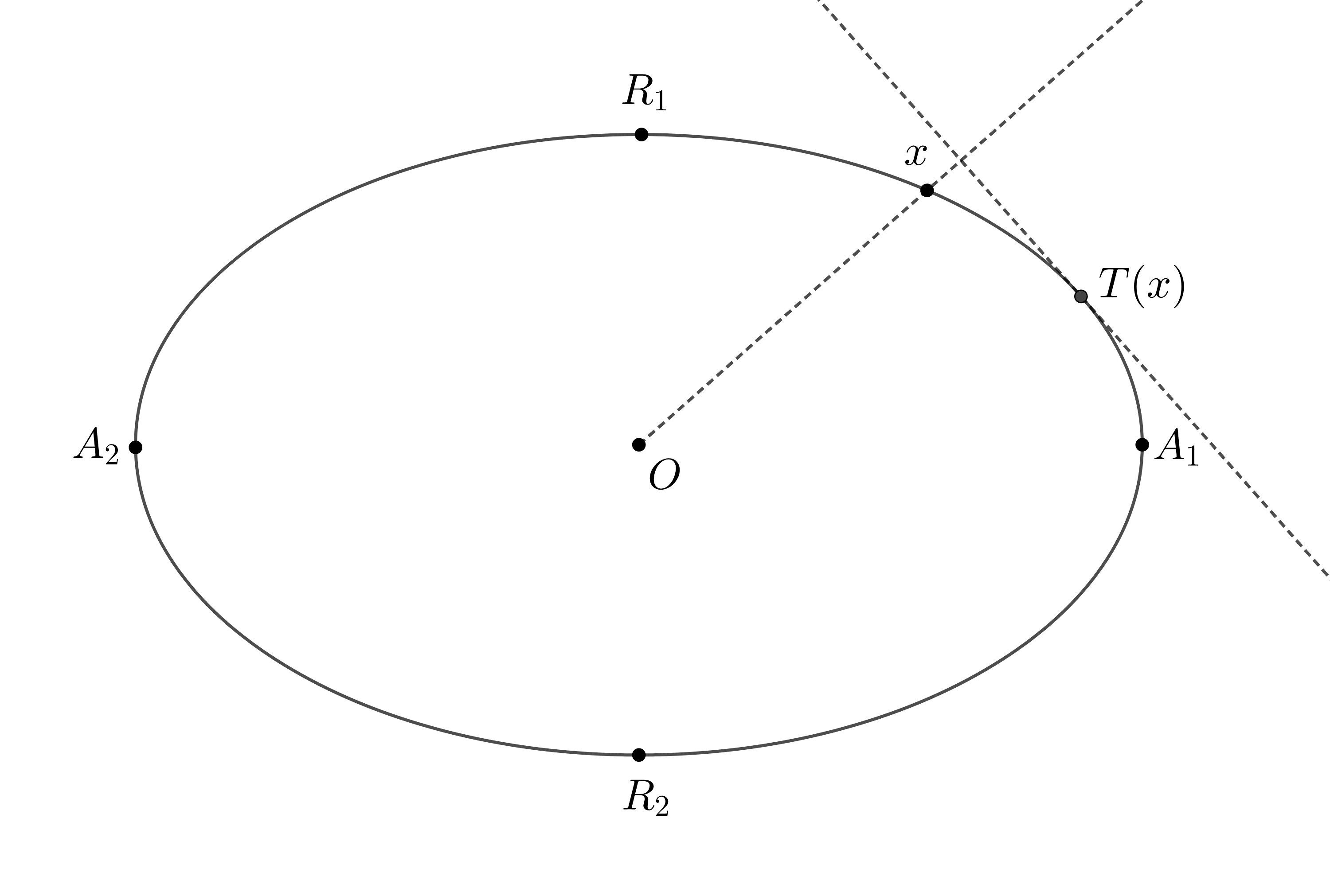}}
      \caption{The fixed points of $T$ in an ellipse.  Here the
        ellipse is centered at the origin $O$.  There are two attractive
        $(A_1,A_2)$ and two repelling $(R_1,R_2)$ fixed points.}
  \label{fig:ellipse}
\end{figure}

In Example~\ref{ex:disk} and Example~\ref{ex:ellipse} above we see
that iteration with $T$ converges towards a fixed point, but never
reaches it in a finite number of steps (unless we start at the fixed
point).  This is always the case when $\Delta$ is a region with smooth
boundary.  For such regions the normal cones are one dimensional and
for $x$ to be a fixed point it must be that $x \in N(\Delta,x)$.
Since no other boundary point can be in the normal cone at $x$, fixed
point iteration cannot reach $x$ in a finite number of steps.

On the other hand, consider the case when $\Delta$ is a polytope.  In
this case fixed point iteration with $T$ always reaches a fixed point
in a finite number of steps.  For a generic point $x \in \Delta$,
$T(x)$ is a vertex of the polytope, and iteration with $T$ defines a
sequence of vertices with increasing norms.  The set of vertices that
are fixed points depends on the position of the origin.  Fixed points
can also exist in higher dimensional faces, but are then never
attractive.

\begin{figure}
  \centerline{\includegraphics[height=2.7in]{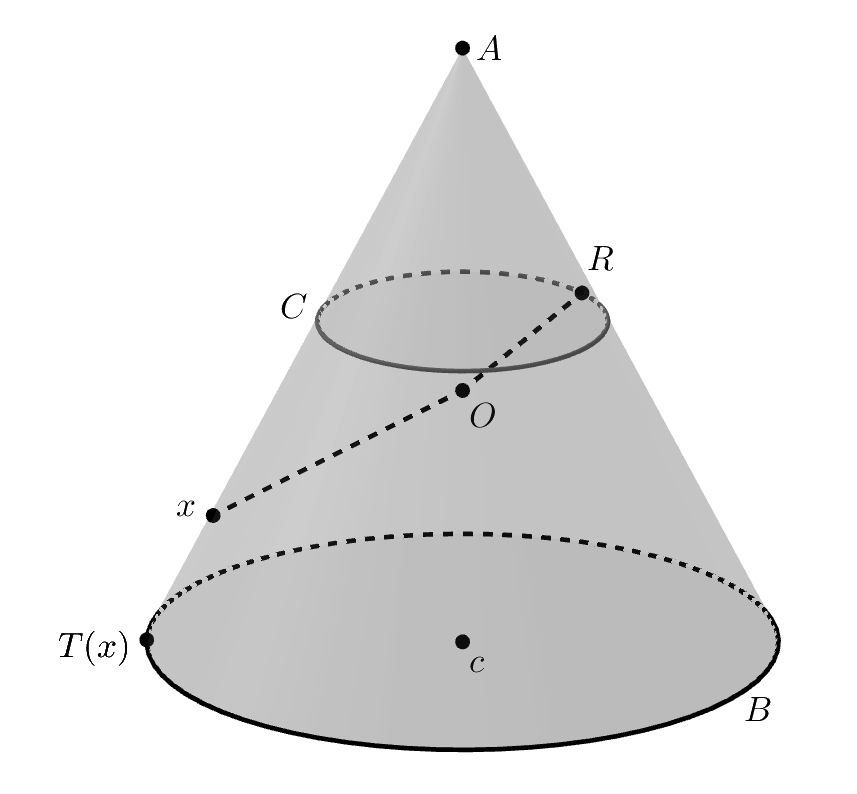}}
  \caption{The fixed points of $T$ in a cone.  Here $O$ is the origin
    while $c$ is the center of the base.  The top of the cone is an
    attractive fixed point $A$.  There is a circle of repelling fixed
    points $C$ in the middle of the cone.  There is also a circle $B$
    of fixed points at the base.  Any point above $C$ maps to $A$
    under $T$, while any point below $C$ maps to $B$.  Fixed points in
    the base are not individually attractive or repelling, but
    together they form an attractive set.}
    \label{fig:cone}
\end{figure}

\begin{example}[Cone]
Figure~\ref{fig:cone} shows an example when $\Delta$ is a
three-dimensional cone.  In this case there is an attractive fixed
point $A$ at the top of the cone.  There is a circle $C$ in the middle
of the cone and each point in $C$ is a repelling fixed point.  We also
have a circle $B$ of fixed points at the base.  Any point above $C$
maps to $A$ under $T$ in a single step.  Similarly any point below $C$
maps to $B$ in a single step.  The map $T$ takes any point near $B$ to
$B$ so we can see $B$ as an attractive set.  However, the points in
$B$ are not individually attractive or repelling.
\end{example}

\subsection{Fixed Point Classification}
\label{sec:types}

Fixed point iteration defines a discrete dynamical system.  There are
several notions of stability for such systems and the notions we use
throughout the paper are defined below (see, e.g.,
\cite{Galor,Holmgren}).

\begin{definition}
A fixed point $x$ is \emph{attractive} if $\exists \epsilon > 0$ such
that $||x-x_0|| < \epsilon$ implies that iteration with $T$ starting
at $x_0$ converges to $x$.
\end{definition}

\begin{definition}
A fixed point $x$ is \emph{repelling} if $\exists \epsilon > 0$ such
that $||x-x_0|| < \epsilon$ and $x_0 \neq x$ implies there is an $n$
for which iteration with $T$ starting at $x_0$ leads to $x_n$ with
$||x_n-x|| > \epsilon$.
\end{definition}

Consider a two-dimensional region $\Delta$ with smooth boundary.  For
any fixed point $x$ the tangent at $x$ is perpendicular to $x$.
Therefore, the behavior of $T$ at a point $y$ sufficiently near $x$
depends only the curvature, $k(x)$, of the boundary at $x$.  If $k(x)
> 1/||x||$ then $x$ is attractive.  In this case the behavior of $T$
near $x$ is similar to the behavior of $T$ near the attractive fixed
point in the off-center disk in Example~\ref{ex:disk}.  On the other
hand, if $k(x) < 1/||x||$ then $x$ is repelling.  In this case the
behavior of $T$ near $x$ is similar to the behavior of $T$ near the
repelling fixed point in the off-center disk in Example~\ref{ex:disk}.

In higher dimensions the situation is more subtle because the the
boundary has multiple principal curvatures.  In particular a fixed
point may be behave like an attractive fixed point along one direction
and like a repelling fixed point along another.  A fixed point $x$ is
attractive if all curvatures at $x$ are greater than $1/||x||$.
Similarly a fixed point is repelling if all curvatures at $x$ are
smaller than $1/||x||$.

\section{Elliptope}
\label{sec:elliptope}

For the remainder of the paper we study the fixed points and the fixed
point iteration process in the special case of the elliptope, the
study of which is motivated by SDP relaxations of combinatorial
optimization problems (see, e.g., \cite{Goemans,Laurent}).  We first
prove an eigenvector-like relationship between points $x$ and $T(x)$.
Theorem~\ref{thm:mfp} builds on this relationship giving an algebraic
characterization of the fixed points in the elliptope.  In
Section~\ref{sec:elliptopeEX}, we fully illustrate all of the fixed
points in dimension $3$ and give an explicit construction of an
infinite family of fixed points in dimension $4$.  In terms of the
iteration process, we classify the attractive fixed points of the
elliptope as exactly its vertices in Theorem~\ref{thm:attractive}.
Finally, we discuss how fixed point iteration can be used to
approximately solve the closest vertex problem and to round the
solutions of SDP relaxations.

Let $\Sn  \subset \reals^{n \times n}$ be the set of $n$ by $n$ symmetric matrices. 

\begin{definition}
The \emph{elliptope}
$\Ln$ is the subset of matrices in $\Sn$ that are
positive semidefinite and have all $1$'s on the diagonal:
$$\Ln = \{ X \in \Sn \,|\, X \succcurlyeq 0, X_{i,i} =
1 \}.$$
\end{definition}

For a matrix $X \in {\cal L}_3$, $X$ has the form, 
$$X =
\begin{pmatrix}
1 & x & y \\
x & 1 & z \\
y & z & 1
\end{pmatrix}.$$
Therefore we can visualize $X$ as a point $(x,y,z) \in \mathbb{R}^3$.

Figure~\ref{fig:elliptope} shows the elliptope ${\cal L}_3$ and the
fixed points of $T$.  The red fixed points are irreducible matrices
with rank 1 and correspond to the vertices of the elliptope, the blue
points are irreducible matrices (see Definition~\ref{def:irreducible})
with rank 2 and the green points are reducible matrices with rank 2.
Example~\ref{ex:l3} in Section~\ref{sec:elliptopeEX} describes the
fixed points of ${\cal L}_3$ in more detail.

As seen in Figure~\ref{fig:elliptope}, the fixed points of ${\cal
  L}_3$ and their ranks are related to the geometry of the convex
body.  In general, the vertices of the elliptope are always fixed
points of $T$.  However, there are other fixed points that reflect
different geometrical structure.  The geometry of the elliptope and
the nature of the fixed points becomes more complex in higher
dimensions.  For example, while the number of fixed points is finite
for $n=3$, when $n>3$ there are already an infinite number of fixed
points.

\subsection{Fixed points in $\Ln$}
\label{sec:elliptopeFP}

It is well known that the matrices in $\Ln$ are precisely the Gram
matrices of $n$ unit vectors in $\reals^n$ (\cite{laurent1997}).  For
example, for the matrix $X$ above, there must exist vectors $v_1, v_2,
v_3 \in \mathbb{R}^3$ with $||v_i|| = 1$, such that $x = v_1^Tv_1$, $y
= v_2^Tv_2$, $z = {v_3}^T v_3$.  Considering the optimization defined
by $T(M)$ and using the Gram matrix representation for the matrices in
$\Ln$ we obtain the following result.

\begin{lemma}
  \label{lem:local}
  Let $M \in \Sn$ and $X=T(M)$.  
  
  Suppose $X$ is the Gram matrix of $n$ unit vectors $(v_1,\ldots,v_n)$.  Then,
  \begin{itemize}
  \item[(a)]There exists real values $\alpha_i$ such that 
  $$\sum_{j \neq i} M_{i,j} v_j = \alpha_i v_i.$$
  \item[(b)]The vectors $(v_1,\ldots,v_n)$ are linearly dependent and $\rank(X) < n$.
  \item[(c)] There exists a diagonal matrix $D$ such that,
    $$MX = DX.$$
  \end{itemize}
\end{lemma}

\begin{proof}
  For $n$ unit vectors $(u_1,\ldots,u_n)$ let $E(u_1,\ldots,u_n) = Y
  \cdot M$ where $Y$ is the Gram matrix of $(u_1,\ldots,u_n)$.  For a
  single unit vector $u$ let
  $$E_i(u) = E(v_1,\ldots,v_{i-1},u,v_{i+1},\ldots,v_n).$$
  Since $X=T(M)$ the unit vectors $(v_1,\ldots,v_n)$ maximize $E$.
  Therefore $v_i$ maximizes $E_i$.  Using the method of Lagrange
  multipliers to maximize $E_i(u)$ subject to $||u||^2=1$ we see that
  $$\nabla E_i(v_i) = \lambda_i \nabla ||v_i||^2 \Leftrightarrow 2 \sum_{j
    \neq i} M_{i,j} v_j = 2 \lambda_i v_i.$$ This proves part (a).

  For part (b) note that if $\lambda_i \neq 0$ then $v_i$ is in the
  span of $\{v_j \,|\, j \neq i\}$.  On the other hand, if $\lambda_i
  = 0$ then $\{v_j \,|\, j \neq i\}$ are linearly dependent.  Let $V$
  be the matrix with $v_i$ in the $i$-th row.  Since $X=VV^T$ we have
  $\rank(X) < n$.

  For part (c) let $D$ be the diagonal matrix where $D_{i,i} =
  \alpha_i+M_{i,i}$.  By part (a) we have $MV=DV$.  Multiplying by
  $V^T$ on both sides we obtain $MX=DX$.
\end{proof}

The relationship between $X=T(M)$ and $M$ defined by $MX = DX$ is
similar to the notion of an eigenvector.  A vector $v$ is an
eigenvector of $M$ with eigenvalue $\lambda$ if $Mv = \lambda v$.  The
condition $MX = DX$ is analogous but we have a matrix $X$ instead of a
vector $v$, and a diagonal matrix $D$ instead of a scalar $\lambda$.
Although this notion of an ``eigenmatrix'' is natural it does not seem
to appear in the literature before.

The next result is one direction of Theorem \ref{thm:mfp}.  We present
this result now because it will be used for some of the intermediate
results leading to the other direction.

\begin{proposition}
  Let $X$ be a fixed point of $T$.
  Then $$X^2 = DX$$ where $D$ is a diagonal matrix with,
  $$D_{i,i} = \sum_j (X_{i,j})^2 \ge 1.$$ 
\label{prop:fp}
\end{proposition}
\begin{proof}
  Lemma~\ref{lem:local} implies $X^2=DX$.  Since $X \in \Ln$ we know
  $X_{i,i} = 1$.  Therefore
  $$D_{i,i} = (DX)_{i,i} = (X^2)_{i,i} = \sum_{j} (X_{i,j})^2 \ge (X_{i,i})^2 = 1.$$ 
\end{proof}

\begin{definition}
\label{def:irreducible}
A matrix $M \in \Sn$ is \emph{irreducible} if we cannot partition
$\{1,\ldots,n\}$ into two sets $A$ and $B$ with $M_{i,j} = 0$ whenever
$i \in A$ and $j \in B$.
\end{definition}

\begin{proposition}
Let $X \in \Ln$ be an irreducible matrix with $X^2=DX$.  Then
$D=\gamma I$ with $\gamma \ge 1$.
\label{obs:irred}
\end{proposition}
\begin{proof}
As in Proposition~\ref{prop:fp} we know $D_{i,i} \ge 1$.  For $i \neq
j$,
$$\frac{1}{D_{i,i}} \sum_{l=1}^n X_{i,l} X_{l,j} = X_{i,j} = X_{j,i} = 
\frac{1}{D_{j,j}} \sum_{l=1}^n X_{j,l} X_{l,i}.$$
If $X_{j,i} = X_{i,j} \neq 0$ then $D_{i,i} = D_{j,j}$.  Since $X$
is irreducible this implies $D = \gamma I$. 
\end{proof}

As an example consider the irreducible matrix $X$ and diagonal matrix $D$,
$$X = \begin{pmatrix*}[r]
1 & -1 & -1 \\
-1 & 1 & 1 \\
-1 & 1 & 1
\end{pmatrix*},\;
D = \begin{pmatrix*}[r]
\frac{1}{3} & 0 & 0 \\
0 & \frac{1}{3} & 0 \\
0 & 0 & \frac{1}{3}
\end{pmatrix*}.$$
In this case $X^2=DX$ and $D=\frac{1}{3}I$.

Now consider the reducible matrix $X$ and diagonal matrix $D$ below,
$$X = \begin{pmatrix*}[r]
1 & 0 & 0 \\
0 & 1 & 1 \\
0 & 1 & 1
\end{pmatrix*},\;
D = \begin{pmatrix*}[r]
1 & 0 & 0 \\
0 & \frac{1}{2} & 0 \\
0 & 0 & \frac{1}{2}
\end{pmatrix*}.$$
In this case $X^2=DX$ but $D\neq \gamma I$ for any $\gamma$.

To characterize the fixed points of $T$ we use a result from
\cite{Laurent} about the normal cones in $\Ln$.

\begin{proposition}[Proposition 2.3 in \cite{Laurent}]
A matrix $Y$ is in the normal cone of $\Ln$ at $X$ if and only if
$Y=D-M$ where $D$ is a diagonal matrix and $X \cdot M = 0$
\label{prop:laurent}
\end{proposition}

Note that the condition $X \cdot M = 0$ is equivalent to
$M=\sum_{j=1}^p w(j)w(j)^\top$ with $w(j) \in \ker(X)$.

\begin{lemma}
Let $X$ be an irreducible matrix in $\Ln$ with rank $s$. Then $X$ is a
fixed point of $T$ if and only if $X$ can be written as
$$X = \frac{n}{s} \sum_{i=1}^s v(i)v(i)^\top,$$ where
$\{v(1),\ldots,v(s)\}$ is an orthonormal set of vectors such that for
all $j = 1,\ldots,n$
$$\frac{n}{s} \sum_{i=1}^s v(i)_j ^2 = 1.$$
\label{lem:ifp}
\end{lemma}
\begin{proof}
First, suppose that $X$ is a fixed point. Let $\{v(1),\ldots,v(s)\}$
be an orthonormal set of eigenvectors for $X$. Since $X$ is symmetric,
we can write $X$ as
$$X = \sum_{i=1}^s \lambda_i v(i)v(i)^\top$$ where $\lambda_i > 0$ is
the eigenvalue associated with the eigenvector $v(i)$. Further, since
$X$ is a fixed point, using Propositon~\ref{prop:laurent}, we can
write $X = D-M$, where $D$ is a diagonal matrix and $M = \sum_{j=1}^p
w(j) w(j)^\top$ with $w(j) \in \mathrm{ker}(X)$. Thus, for all $i =
1,\ldots,s$
$$\lambda_i v(i) = X v(i) = (D-M)v(i) = D v(i).$$ This shows that
$D_{j,j} = \lambda_i$ for all $j$ such that $v(i)_j \neq 0$. Since $X$
is irreducible, this implies that $\lambda_i = \lambda$ for all
$i=1,\ldots,s$ and $D = \lambda I$.  Further, for all $j = 1,\ldots,n$
$$1 = X_{j,j} = \lambda \sum_{i=1}^s v(i)_j^2.$$ Summing over all $j$,
yields $n = \lambda s$ or $\lambda = n/s$. This completes the first
direction of the proof.

Now suppose that $X$ can be written as
$$X = \frac{n}{s} \sum_{i=1}^s v(i)v(i)^\top,$$ as above.  Consider
the matrix $M = \frac{n}{s} I - X$. Since $M$ is symmetric, we can
write $M = \sum_{j=1}^p \alpha_j w(j) w(j)^\top$, where
$\{w(1),\ldots,w(p)\}$ is a set of orthonormal eigenvectors of $M$ and
$\alpha_j$ is the eigenvalue associated with $w(j)$.  For all
$j=1,\ldots,p$
$$\alpha_j w(j) = M w(j) = \frac{n}{s} w(j) - X w(j).$$ Thus, $w(j)$
is an eigenvector of $X$ as well. This shows that either $\alpha_j =
n/s$ and $w(j)$ is in the kernel of $X$ or $\alpha_j = 0$ (and we can
remove these vectors from the sum defining $M$).  Now
Proposition~\ref{prop:laurent} implies $X$ is a fixed point.
\end{proof}

We now prove our main result of this Section which characterizes the
set of fixed points in $\Ln$.

\begin{theorem}
  Let $X$ be a matrix in $\Ln$.  Then $X$ is a fixed point of $T$ if
  and only if $$X^2=DX$$ where $D$ is a diagonal matrix.
\label{thm:mfp}
\end{theorem}
\begin{proof}
First consider the case where $X$ is an irreducible irreducible matrix
in $\mathcal{L}_n$. Then we claim that $X$ is a fixed point if and
only if $X^2=DX$ where $D$ is a diagonal matrix.
 
When $X$ is a fixed point Proposition~\ref{prop:fp} implies $X^2 =
DX$.

Now suppose $X^2 = DX$.  Since $X$ is irreducible $D=\gamma I$ with
$\gamma \ge 1$.  Let $\{v(1),,\ldots,v(s)\}$ be an orthonormal set of
eigenvectors for $X$. Since $X$ is symmetric, we can write $X$ as
$$X = \sum_{i}^s \lambda_i v(i)v(i)^\top$$
where $\lambda_i > 0$ is the eigenvalue associated with the eigenvector $v(i)$. For all $i = 1,2, \ldots, s$
$$\lambda_i v(i)  = X v(i) = \frac{1}{\gamma} X^2 v(i) = \frac{\lambda_i^2}{\gamma} v(i).$$
Thus, $\lambda_i = \lambda$ for all $i$ with $\lambda = \gamma$.  
For $j = 1, 2, \ldots, n$ 
$$1 = X_{j,j} = \lambda \sum_{i=1}^s v(i)_j^2.$$
Summing over all $j$, yields $n = \lambda s$ or $\lambda = n/s$.  Now Lemma~\ref{lem:ifp}
implies $X$ is a fixed point.

Next suppose $X$ is a fixed point that does not necessarily correspond
to an irreducible matrix.  Then $X^2=DX$ by Proposition~\ref{prop:fp}.

Now suppose $X^2=DX$.  For $A \subseteq \{1,\ldots,n\}$ and $M \in
\Sn$ let $M|_A$ be the submatrix of $M$ indexed by the rows and
columns in $A$.  Let $G=(V,E)$ be a graph with $V=\{1,\ldots,n\}$ and
$E =\{\{i,j\} \,|\, X_{i,j} \neq 0\}$.  Let $\{A_1,\ldots,A_k\}$ be
the connected components of $G$.  Then each submatrix $X|_{A_i}$ is
irreducible and $X_{r,s} = 0$ if $r \in A_i$ and $s \in A_j$ with $i
\neq j$.  Each irreducible block of $X$ is square, symmetric, positive
semidefinite, and has $1$'s in the diagonal.  Therefore $X|_{A_i} \in
{\mathcal L}_{|A(i)|}$.

Since $X^2=DX$ we have $(X|_{A_i})^2 = (D|_{A_i}) (X|_{A_i})$.  Since
$X|_{A_i}$ is irreducible, $X|_{A_i}$ is a fixed point in ${\mathcal
  L}_{|A_i|}$.  We can use Proposition~\ref{prop:laurent} to write
$X|_{A_i} = L(i) - M(i)$ where $L(i)$ is diagonal and $X|_{A_i} \cdot
M(i) = 0$.  Let $L$ be the $n \times n$ matrix with $L|_{A_i} = L(i)$
and zeros in other entries.  Let $M$ be the $n \times n$ matrix with
$M|_{A_i} = M(i)$ and zeros in other entries.  Then $X = L - M$ with
$L$ diagonal and $X \cdot M = 0$.  Now Proposition~\ref{prop:laurent}
implies $X$ is a fixed point.
\end{proof}

\subsection{Examples}
\label{sec:elliptopeEX}

Here we consider two examples that illustrate the fixed points of $T$
in elliptopes of different dimensions.  Example~\ref{ex:l3} describes
all of the fixed points in ${\cal L}_3$.

\begin{example}
\label{ex:l3}
Figure~\ref{fig:elliptope} illustrates the fixed points in ${\cal
  L}_3$.  In this case there are finitely many fixed points.  There
are 4 irreducible fixed points with rank 1 corresponding to the
vertices of ${\cal L}_3$ (shown as red points in
Figure~\ref{fig:elliptope}).  The corresponding matrices are:
$$\begin{pmatrix*}[r]
1 & -1 & -1 \\
-1 & 1 & 1 \\
-1 & 1 & 1
\end{pmatrix*},
\begin{pmatrix*}[r]
1 & -1 & 1 \\
-1 & 1 & -1 \\
1 & -1 & 1
\end{pmatrix*},
\begin{pmatrix*}[r]
1 & 1 & -1 \\
1 & 1 & -1 \\
-1 & -1 & 1
\end{pmatrix*},
\begin{pmatrix*}[r]
1 & 1 & 1 \\
1 & 1 & 1 \\
1 & 1 & 1
\end{pmatrix*}.$$
There are 6 reducible fixed points with rank 2 (shown as green points
in Figure~\ref{fig:elliptope}).  Each of these fixed points is the
average of two vertices and appear along an ``edge'' of ${\cal L}_3$.
The corresponding matrices are:
$$\begin{pmatrix*}[r]
1 & 1 & 0 \\
1 & 1 & 0 \\
0 & 0 & 1
\end{pmatrix*},
\begin{pmatrix*}[r]
1 & -1 & 0 \\
-1 & 1 & 0 \\
0 & 0 & 1
\end{pmatrix*},
\begin{pmatrix*}[r]
1 & 0 & 1 \\
0 & 1 & 0 \\
1 & 0 & 1
\end{pmatrix*},
\begin{pmatrix*}[r]
1 & 0 & -1 \\
0 & 1 & 0 \\
-1 & 0 & 1
\end{pmatrix*},
\begin{pmatrix*}[r]
1 & 0 & 0 \\
0 & 1 & 1 \\
0 & 1 & 1
\end{pmatrix*},
\begin{pmatrix*}[r]
1 & 0 & 0 \\
0 & 1 & -1 \\
0 & -1 & 1
\end{pmatrix*}.$$
Finally, there are 4 irreducible fixed points with rank 2, one for
each ``puffed face'' in ${\cal L}_3$ (shown as blue points in
Figure~\ref{fig:elliptope}).  Each of these fixed points equals $T(M)$
for a matrix $M$ that is the average of 3 vertices, the average itself
does not lie on the boundary of ${\cal L}_3$.  The corresponding
matrices are:
$$\begin{pmatrix*}[r]
1 & -\frac{1}{2} & -\frac{1}{2} \\
-\frac{1}{2} & 1 & -\frac{1}{2} \\
-\frac{1}{2} & -\frac{1}{2} & 1
\end{pmatrix*},
\begin{pmatrix*}[r]
1 & -\frac{1}{2} & \frac{1}{2} \\
-\frac{1}{2} & 1 & \frac{1}{2} \\
\frac{1}{2} & \frac{1}{2} & 1
\end{pmatrix*},
\begin{pmatrix*}[r]
1 & \frac{1}{2} & -\frac{1}{2} \\
\frac{1}{2} & 1 & \frac{1}{2} \\
-\frac{1}{2} & \frac{1}{2} & 1
\end{pmatrix*},
\begin{pmatrix*}[r]
1 & \frac{1}{2} & \frac{1}{2} \\
\frac{1}{2} & 1 & -\frac{1}{2} \\
\frac{1}{2} & -\frac{1}{2} & 1
\end{pmatrix*}.$$
\end{example}

While there are a finite number of fixed points in ${\cal L}_3$, there
is an infinite set of fixed points in $\Ln$ for $n > 3$.
Example~\ref{ex:l4} illustrates an infinite family of fixed points in
${\cal L}_4$.

\begin{example}
\label{ex:l4}
In ${\cal L}_4$, any value $-1 < c < 1$ leads to a distinct fixed
point,
$$X(c) = \begin{pmatrix*}[r]
1 & -\sqrt{1-c^2} & 0 & c \\
-\sqrt{1-c^2} & 1 & -c & 0 \\
0 & -c & 1 & -\sqrt{1-c^2} \\
c & 0 & -\sqrt{1-c^2} & 1 
\end{pmatrix*}.$$
In this case we have $X(c)=\frac{1}{2}X(c)^2$.
\end{example}

Although there can be an infinite number of fixed points in
$\mathcal{L}_n$, there can only be a finite number of regular points
which are fixed points, where a point is regular (or smooth) if it has
a one-dimensional normal cone \cite{Laurent}.

\begin{lemma}
  In $\Ln$ there is a finite number of regular points that are also
  fixed points.
\end{lemma}
\begin{proof}
By definition, a regular point is an extreme point whose normal cone
is 1-dimensional.  By Proposition~\ref{prop:laurent} the kernel of $X$
is 1-dimensional. Let $w$ be the eigenvector of $X$ with eigenvalue 0
scaled such that $|| w || = 1$. Then we can write $X$ as
$$X = D-M$$
where $M = \alpha w w^\top$, $\alpha > 0$, and $D$ is a diagonal matrix. Then, 
\begin{align*}
\vec{0} & = Xw = (D-M) w = (D - \alpha ww^\top) w = (D-\alpha I) w,
\end{align*}
where the last equality holds since $w^T w = 1$.  Thus, if $w_i \neq
0$, then $D_{i,i} = \alpha$.  Further, since $X_{i,i} = 1$ we know
that $M_{i,i} = \alpha- 1$ if $w_i \neq 0$.  We also know that
$M_{i,i} = \alpha w_i^2$. Thus, either $w_i = 0$ or $w_i = \pm
\sqrt{1-1/\alpha}$, where $\alpha$ is set so that $||w|| =
1$.

This shows that the fixed points are a subset of those matrices whose
kernel is spanned by some $w \in \{0, 1,-1\}^n$. In particular, set $w
\in \{0, 1,-1\}^n$ and $M = \alpha w w^\top$ for some
$\alpha$. Suppose that $w$ has $p$ non-zero entries. Then, $X \cdot M
= 0$ and $X=D-M$ imply that $p \alpha (1+\alpha) = p^2 \alpha^2 $ or
$\alpha = 1/(p-1)$. This shows the unique construction of $X$ from
$w$.
\end{proof}

\subsection{Iteration in $\Ln$ and the closest vertex problem}
\label{sec:elliptopeITER}

An important step in using a convex relaxation to solve a
combinatorial optimization problem involves \emph{rounding} a point
$X$ in the convex body to an \emph{integer solution} $Y$ that is
feasible for the underlying combinatorial problem.  In the classical
SDP relaxation of max-cut the integer solutions are the $\{-1,+1\}$
symmetric matrices in $\Ln$ of rank 1 (\cite{Goemans}).  An integer
solution $Y$ defines a partition of $[n]$ into two sets, where $i$ is
in the same set as $j$ if $Y_{i,j} = 1$ and in a different set if
$Y_{i,j}=-1$ The integer solutions for max-cut are exactly the
vertices of $\Ln$ (\cite{Laurent}) and one can solve the rounding
problem for the SDP relaxation by finding the closest vertex to a
matrix $X \in \Ln$.

Fixed point iteration with $T$ defines a deterministic method for
solving the rounding problem.  In particular, fixed point iteration
solves a sequence of relaxations to the closest vertex problem.
Moreover, the vertices of $\Ln$, which define partitions, are
precisely the attractive fixed points of $T$.

First note that $$||X-Y||^2 = X \cdot X + Y \cdot Y - 2 (X \cdot Y).$$ 

For a vertex $Y$, $Y \cdot Y = n^2$.  Therefore we can find the vertex
$Y$ that is closest to $X$ by maximizing $X \cdot Y$.  Relaxing this
problem to $\Ln$ gives an SDP relaxation to the closest vertex
problem, defined by $Y=T(X)$.

If $Y=T(X)$ is vertex, then $Y$ is is the closest vertex
to $X$.  On the other hand, if $Y$ is not a vertex, we consider
(recursively) the problem of finding the vertex $Z$ that is
closest to $Y$.  This involves iterating $T$ to compute $Z=T(Y)$.
Thus fixed point iteration solves a 
sequence relaxations of the closest vertex problem.  

To show the vertices are the only attractive fixed points we need the
following Lemma.

\begin{proposition}
Let $X$ be any fixed point that is not a vertex of $\Ln$.  Then, there
exists a curve $\hat{X}(\alpha)$ $(0 \le \alpha \le 1)$ such that
$\hat{X}(0) = X$ and $\hat{X}(\alpha) \cdot \hat{X}(\alpha) > X \cdot
X$ when $\alpha > 0$.
\label{prop:increase}
\end{proposition}
\begin{proof}
Since $X$ is not a vertex, there exists $i \neq j$ such that $X_{i,j}
\not \in \{-1,+1\}$ and $\sum_{l} (X_{i,l})^2 \leq \sum_{l}
(X_{j,l})^2$.  Suppose $X$ is the Gram matrix of $\{v_1, \ldots,
v_n\}$.  We will construct $\hat{X}(\alpha)$ by moving the vector
$v_i$ towards either $v_j$ or $-v_j$.

We first consider the case that $X_{i,j} \geq 0$.  In this case, we
move $v_i$ towards $v_j$.  For $0 \leq \alpha \leq 1$, define the unit
vector
$$\hat{v}_i (\alpha) = z_{\alpha}((1-\alpha) v_i + \alpha v_j),$$
where $z_{\alpha} = 1/\sqrt{(1-\alpha)^2 + \alpha^2 + 2 \alpha
  (1-\alpha) X_{i,j}}>1$ is a normalization factor.  Now define
$\hat{X}(\alpha)$ to be the Gram matrix of $\{v_1, \ldots, v_{i-1},
\hat{v}_i(\alpha), v_{i+1}, \ldots, v_n \}$. Note that for $l \neq i$,
$$\hat{X}_{i,l}(\alpha) = z_{\alpha} ((1-\alpha) X_{i,l} + \alpha X_{j,l}).$$

We now analyze the value of $\hat{X}(\alpha) \cdot \hat{X}(\alpha)$
and compare it to $X \cdot X$.
\begin{align*}
\hat{X}(\alpha) \cdot \hat{X}(\alpha) & = X \cdot X + 2 \sum_{l \neq i} (\hat{X}(\alpha)_{i,l})^2 - 2 \sum_{l \neq i} (X_{i,l})^2,
\end{align*}
where
\begin{align*}
   \sum_{l \neq i} (\hat{X}(\alpha)_{i,l})^2 &= \sum_{l \neq i} \frac{ (1-\alpha)^2 (X_{i,l})^2 + \alpha^2 (X_{j,l})^2 + 2 \alpha(1-\alpha) X_{i,l} X_{j,l}}{(1-\alpha)^2 + \alpha^2 + 2 \alpha (1-\alpha) X_{i,j}} \\
   & = \frac{(1-\alpha)^2 \sum_{l \neq i} (X_{i,l})^2 + \alpha^2 \sum_{l \neq i} (X_{j,l})^2 + 2 \alpha (1-\alpha) \sum_{l \neq i} X_{i,l} X_{j,l}}{(1-\alpha)^2 + \alpha^2 + 2 \alpha (1-\alpha) X_{i,j}}.
\end{align*}
Since $X$ is a fixed point, we know that $X^2 = DX$ with $D_{i,i} =
\sum_l (X_{i,l})^2$, which implies that
$$X_{i,j} = \frac{ \sum_{l} X_{i,l} X_{j,l}}{\sum_{l} (X_{i,l})^2}.$$
Rearranging this expression,
$$\sum_{l \neq i} X_{i,l} X_{j,l} = X_{i,j} \sum_{l} (X_{i,l})^2 - X_{i,j} = X_{i,j} \sum_{l \neq i} (X_{i,l})^2.$$
We can now see that
\begin{align*}
\sum_{l \neq i} (\hat{X}(\alpha)_{i,l})^2 & = \frac{(1-\alpha)^2 \sum_{l \neq i} (X_{i,l})^2 + \alpha^2 \sum_{l \neq i} (X_{j,l})^2 + 2 \alpha (1-\alpha) X_{i,j} \sum_{l \neq i} (X_{i,l})^2 }{(1-\alpha)^2 + \alpha^2 + 2 \alpha (1-\alpha) X_{i,j}}.
\end{align*}
Last, we note that 
$$\sum_{l \neq i} (X_{j,l})^2 = \sum_{l \neq j} (X_{j,l})^2 + 1 - (X_{i,j})^2.$$
Overall, this shows that
$$\sum_{l \neq i} (\hat{X}(\alpha)_{i,l})^2 \geq 
\frac{(1-\alpha)^2 \sum_{l \neq i} (X_{i,l})^2 + \alpha^2 [1- (X_{i,j})^2 + \sum_{l \neq j} (X_{j,l})^2] + 2 \alpha (1-\alpha) X_{i,j} \sum_{l \neq i} (X_{i,l})^2 }{(1-\alpha)^2 + \alpha^2 + 2 \alpha (1-\alpha) X_{i,j}}$$
and $\hat{X} (\alpha) \cdot \hat{X} (\alpha) > X \cdot X$ if $\alpha \neq 0$.
Moreover, $\hat{X}(\alpha) \cdot \hat{X} (\alpha)$ strictly increases as $\alpha$ increases.

When $X_{i,j} < 0$ a similar argument leads to the desired curve if we
move $v_i$ towards $-v_j$.
\end{proof}

\begin{theorem}
The vertices of $\Ln$ are the attractive fixed points of $T$.
\label{thm:attractive}
\end{theorem}
\begin{proof}
Let $X$ be a vertex of $\Ln$ and $M \in \Ln$ with $||M-X|| < 1$.  Then
$|X_{i,j}-M_{i,j}| < 1$.  Since $X_{i,j} \in \{-1,+1\}$ the matrix $M$
has the same sign pattern as $X$.  That is, for every $+1$ entry in
$X$ the corresponding entry in $M$ is positive, and for $-1$ entry in
$X$ the corresponding entry in $M$ is negative.  If $Y \in \Ln$ then
$|Y_{i,j}| \le 1$.  Therefore $Y \neq X \Rightarrow M \cdot X > M
\cdot Y$.  We conclude $T(M) = X$ and fixed point iteration from $M$
converges to $X$ in a single step.

Now suppose $X$ is not a vertex.  Proposition~\ref{prop:increase}
implies $\forall \epsilon > 0$ $\exists Y$ with $||Y-X|| < \epsilon$
and $Y \cdot Y > X \cdot X$.  Since $T(Y) \cdot T(Y) \ge Y \cdot Y$,
fixed point iteration from $Y$ cannot converge to $X$.
\end{proof}

In the proof above we show that if $M$ is a matrix with the same sign
pattern as a vertex $X$, then $T(M)=X$.  The set of matrices $M$ for
which $T(M)=X$ was considered in \cite{Cifuentes} (related problems
were also considered in \cite{Thomas}).  More generally we would like
to understand the set $S(X)$ for which fixed point iteration starting
from $S(X)$ converges to $X$.  In ${\cal L}_3$ fixed point iteration
from a generic starting point always converges to the closest vertex.
However, in higher dimensions fixed point iteration can converge to a
vertex that is not closest to the starting point.

\bibliographystyle{plain}
\bibliography{biblio.bib}

\end{document}